\documentclass[final,times,12pt]{elsarticle}

\usepackage{graphicx}
\usepackage{amssymb}
\usepackage{amsthm}
\usepackage{amsmath}
\usepackage{pst-all}
\usepackage{longtable}
\usepackage{subfig}
\newtheorem{theorem}{Theorem}
\newtheorem{proposition}[theorem]{Proposition}
\newtheorem{lemma}[theorem]{Lemma}

\newdefinition{definition}[theorem]{Definition}
\newdefinition{assumption}[theorem]{Assumption}
\newdefinition{remark}[theorem]{Remark}
\newdefinition{algorithm}[theorem]{Algorithm}
\newcommand{\R}{\mathbb{R}}
\newcommand{\N}{\mathbb{N}}
\newcommand{\X}{\mathbb{X}}
\newcommand{\U}{\mathbb{U}}
\def\argmin{\mathop{\rm argmin}}

\usepackage{framed}
{\endMakeFramed}
\definecolor{shadecolor}{rgb}{1.,1.,1.}%
\definecolor{framecolor}{rgb}{.0,.0,.0}%

\journal{System \& Control Letters}

\begin{document}
\begin{frontmatter}

\title{Horizon Adaptation for Nonlinear Model Predictive Controllers with guaranteed Degree of Suboptimality}

\author{J\"{u}rgen Pannek}
\ead{juergen.pannek@googlemail.com}
\ead[url]{www.nonlinearmpc.com}
\address{Curtin University of Technology, Perth, 6845 WA, Australia\fnref{label3}}

\begin{abstract}
	We propose \textit{adaptation strategies} to modify the standard constrained model predictive controller scheme in order to guarantee a certain lower bound on the degree of suboptimality. Within this analysis, the length of the optimization horizon is the parameter we wish to adapt. We develop and prove several shortening and prolongation strategies which also allow for an effective implementation. Moreover, extensions of stability results and suboptimality estimates to model predictive controllers with varying optimization horizon are shown.
\end{abstract}


\end{frontmatter}

\section{Introduction}
Nowadays, nonlinear model predictive controllers (NMPC), sometimes also called receding horizon controllers (RHC) are used in a variety of industrial applications, cf. \cite{QB2003}. As shown in \cite{MRRS2000,RM2009}, theory for such controllers is also widely understood both for linear and nonlinear systems. The control method itself deals with the problem of approximately solving an infinite horizon optimal control problem which is computationally intractable in general. Reasons for its success are on the one hand its capability to directly incorporate constraints depending on the states and inputs of the underlying process. On the other hand, the fundamental steps of this method are very simple: First, a solution of a finite horizon optimal control problem is computed for a given initial value. In a second step, the first part of the resulting control is implemented at the plant and in the third and last step, the finite horizon is shifted forward in time which renders this method to be iteratively applicable. As a consequence, the control which is applied at the plant is a static state feedback.

Due to considering only finite horizons, the inherent stability property of the infinite horizon problem does in general not carry over to the NMPC problem. To cope with the stability issue, several solutions have been proposed in the past, i.e. by imposing endpoints constraints \cite{KG1988} or adding so called Lyapunov function type endpoint weights and a terminal region to the NMPC problem \cite{CA1998}.  A third idea deals with the plain NMPC problem without the requirement of added constraints or a modified cost function. To show stability of the resulting closed loop, in \cite{GP2009, GR2008} a relaxed Lyapunov inequality is assumed.

In either case, the horizon needs to be chosen as a worst case scenario which is usually needed to cope with small regions of state space only. Our aim in this work is to develop online applicable adaptation strategies for the horizon length which guarantee stability of the closed loop. Here, we follow the third approach since the original intention of the infinite horizon cost stays untouched, and make use of the suboptimality estimates given in \cite{GP2009}. Based on the structure of these suboptimality estimates and on the structure of the NMPC problem itself, we propose several techniques to fit the horizon to the control task, the current state of the system and also to the internal information of the NMPC controller itself. Due to the change of the structure of the controller, however, known stability proofs and suboptimality results \cite{GMTT2005, GP2009, GPSW2010, GR2008} cannot be applied. To cover these issues, stability results with varying optimization horizons are presented.

To some extend adaptation strategies of the horizon are known in the literature, see, e.g., \cite{FLJ2006,VBRSB2008}. In contrast to these pure heuristics, our approach can be proven rigorously and doesnot require any insight into the process under consideration. Moreover, a change of the cost functional is possible without modification of the adaptation law which allows for testing various settings of the controller. Last, the quality of the resulting closed loop is tunable by a single variable characterizing the allowable tradeoff compared to the infinite horizon optimal control law.

The paper is organized as follows: In Section \ref{Section:Setup and Preliminaries} we describe the problem setup and state the a posteriori and a priori suboptimalty estimates which will be the foundation of our analysis. In the following Section \ref{Section:Adapting the NMPC Scheme}, we first show how known stability results and estimates can be extended to the case of varying optimization horizons. In Section \ref{Section:Adaptation Strategies}, we develop various shortening and prolongation strategies based on the suboptimality estimates from \cite{GP2009}. To show applicability and effectivity of the proposed methods we present numerical results in Section \ref{Section:Numerical Results}. Finally, Section \ref{Section:Conclusion} concludes the paper and points out directions of future research.

\section{Setup and Preliminaries}
\label{Section:Setup and Preliminaries}

In this work we consider nonlinear discrete time systems of the form
\begin{align}
	\label{Setup:nonlinear discrete time system}
	x(n + 1) = f(x(n), u(n)), \quad x(0) = x_0
\end{align}
with $x(n) \in X$ and $u(n) \in U$ for $n \in \N_0$ where $\N_0$ denotes the natural numbers including zero. In this context, the state space $X$ and the control value space $U$ are arbitrary metric spaces. Therefore, all presented results also apply to the discrete time dynamics induced by a sampled infinite dimensional system. 
State and control constraints can be incorporated by replacing $X$ and $U$ by appropriate subsets $\X \subset X$ and $\U \subset U$. Here, we denote the space of control sequences $u: \N_0 \rightarrow \U$ by $\U^{\N_0}$ and the solution trajectory for given control $u \in \U^{\N_0}$ and initial value $x_0 \in \X$ by $x_u(\cdot, x_0)$.

The task which we pursue is to find a static state feedback $u = \mu(x) \in \U^{\N_0}$ for a given control system \eqref{Setup:nonlinear discrete time system} which minimizes the \textit{infinite horizon cost functional} $J_\infty (x_0, u) = \sum_{n=0}^\infty l(x_u(n, x_0), u(n))$ with stage cost $l: \X \times \U \rightarrow \R_0^+$ where $\R_0^+$ denotes the nonnegative real numbers. The \textit{optimal value function} for this problem is denoted by $V_\infty(x_0) = \inf_{u \in \U^{\N_0}} J_\infty(x_0, u)$. Moreover, one can prove optimality of the \textit{infinite horizon feedback law} $\mu(\cdot)$ given by
\begin{align}
	\label{Setup:infinite control}
	\mu(x(n)) = \argmin_{u \in \U} \left\{ V_\infty(x_u(1, x(n))) + l(x(n), u) \right\}
\end{align}
using Bellman's optimality principle for a given optimal value function. Here we use the $\argmin$ operator in the following sense: for a map $a:\U \to \R$, a nonempty subset $\widetilde{\U} \subseteq \U$ and a value $u^\star \in \widetilde{\U}$ we write 
\begin{align} 
	\label{Setup:argmin}
	u^\star = \argmin_{u \in \widetilde{\U}} a(u)
\end{align} 
if and only if $a(u^\star) = \inf_{u \in \widetilde{\U}} a(u)$ holds. Whenever \eqref{Setup:argmin} holds the existence of the minimum $\min_{u \in \widetilde{\U}} a(u)$ follows. However, we do not require uniqueness of the minimizer $u^\star$. In case of uniqueness equation \eqref{Setup:argmin} can be understood as an assignment, otherwise it is just a convenient way of writing ``$u^\star$ minimizes $a(u)$''. Here we assume that the minimum with respect to $u \in \U$ is attained. 

Since the computation of the desired control law requires the solution of a Hamilton--Jacobi--Bellman equation, we use a model predictive control approach in order to avoid the problem of solving an infinite horizon optimal control problem. The NMPC methodology is simple and consists in three steps which are repeated at every discrete time instant during the process run: Upon start of each iterate, an optimal control for the problem on a finite horizon is computed. Then, the first element of the control is implemented at the process and in the third step the entire optimal control problem considered in the first step is shifted forward in time by one discrete time instant, see, e.g., \cite{MRRS2000} for an overview of this method.

Concerning the computing step of the control law, we consider a finite horizon optimal control problem, that is we minimize the truncated cost functional
\begin{align}
	\label{Setup:finite cost functional}
	J_N(x_0, u) = \sum\limits_{k = 0}^{N - 1} l(x_u(k, x_0), u(k)).
\end{align}
For reasons of clarity, we denote the closed loop solution at time instant $n$ by $x(n)$ throughout this work while $x_u(\cdot, x_0)$ denotes the open loop trajectory of the prediction. Moreover, we use the abbreviation
\begin{align}
	\label{Setup:open loop control}
	u_N(\cdot, x_0) = \argmin_{u \in \U^N} J_N(x_0, u) \quad \text{and} \quad u_N(x_0) = u_N(0, x_0)
\end{align}
for the minimizing \textit{open loop control} sequence of the truncated cost functional and its first element respectively. Moreover, we denote the optimal value function of the finite cost functional \eqref{Setup:finite cost functional} by $V_N(x_0) = \min_{u \in \U^{N}} J_N(x_0, u)$.\\
Given the initial value $x_{u_N}(0, x_0) = x_0$, the open loop control \eqref{Setup:open loop control} induces the \textit{open loop solution}
\begin{align}
	\label{Setup:open loop solution}
	x_{u_N}(k + 1, x_0) = f\left( x_{u_N}(k, x_0), u_N(k, x_0) \right), \quad \forall k \in \{0, \ldots, N-1\}.
\end{align}
Via the implementation and shift steps of the NMPC controller described earlier, we obtain a feedback control $\mu_N(\cdot)$ which can be defined via Bellman's principle of optimality
\begin{align}
	\label{Setup:closed loop control}
	\mu_N(x(n)) = \argmin_{u \in \U} \left\{ V_{N - 1}(x_u(1, x(n))) + l(x(n), u) \right\}.
\end{align}
Using the feedback $\mu_N(\cdot)$, the \textit{closed loop system} is given by
\begin{align}
	\label{Setup:closed loop solution}
	x(n + 1) = f\left( x(n), \mu_N(x(n)) \right), \quad x(0) = x_0, \; n \in \N_0.
\end{align}
In the following, we are interested in the stability and suboptimality properties of the closed loop solution \eqref{Setup:closed loop control}, \eqref{Setup:closed loop solution}. Note that due to the truncation of the infinite horizon cost functional, stability and optimality properties induced by the infinite horizon optimal control \eqref{Setup:infinite control} are not preserved in general. Here, we focus on the NMPC implementation without additional stabilizing endpoint constraints or a Lyapunov function type endpoint costs and a terminal region which are outlined in, e.g., \cite{KG1988} and \cite{CA1998} respectively.

Our aim in this work is to show that the requirement of considering the worst case optimization horizon $N$ for all initial values $x \in \X$ can be weakened without loosing stability of the closed loop \eqref{Setup:closed loop control}, \eqref{Setup:closed loop solution}. Additionally, we show that the resulting closed loop trajectory satisfies locally a predefined degree of suboptimality compared to the infinite horizon solution \eqref{Setup:nonlinear discrete time system}, \eqref{Setup:infinite control} with $u(n) = \mu(x(n))$. To this end, we compare the infinite horizon cost induced by the NMPC control law $\mu_N(\cdot)$, that is $V_\infty^{\mu_N} (x_0) := \sum_{n = 0}^\infty l\left( x(n), \mu_N(x(n)) \right)$, and the finite horizon cost $V_N(\cdot)$ or the infinite horizon optimal value function $V_\infty(\cdot)$. In particular, the latter gives us estimates about the degree of suboptimality of the controller $\mu_N(\cdot)$ of the NMPC process.

Note that since we do not assume terminal constraints to be imposed, feasibility of the NMPC scheme is an issue that cannot be neglected. In particular, without these constraints the closed loop trajectory might run into a dead end. To exclude such a scenario, we assume the following viability condition to hold. We like to note that in case of stabilizing endpoint constraints \cite{KG1988} or a terminal region \cite{CA1998} this assumption holds implicitely.
\begin{assumption}
	For each $x \in \X$ there exists a control $u \in \U$ such that $f(x, u) \in \X$.
\end{assumption}
In order to derive adaptation strategies for the horizon length in this setting, we make extensive use of the suboptimality estimates derived in \cite{GP2009}. Methods to evaluate these estimates rely on a rather straightforward and easily proved ``relaxed" version of the dynamic programming principle, see also \cite{GR2008,LR2006}. For proofs of the following estimate see \cite[Proposition 3]{GP2009}.
\begin{proposition}\label{Preliminaries:prop:trajectory a posteriori estimate}
	Consider a feedback law $\mu_N: \X \rightarrow \U$ and its associated trajectory $x(\cdot)$ according to \eqref{Setup:closed loop solution} with initial value $x(0) = x_0 \in \X$. If there exists a function $V_N: \X \rightarrow \R_0^+$ satisfying
	\begin{align}
		\label{Preliminaries:prop:trajectory a posteriori estimate:eq1}
		V_N(x(n)) \geq V_N(x(n+1)) + \alpha l(x(n), \mu_N(x(n)))
	\end{align}
	for some $\alpha \in [0, 1]$ and all $n \in \N_0$ then
	\begin{align}
		\label{Preliminaries:prop:trajectory a posteriori estimate:eq2}
		\alpha V_{\infty}(x(n)) \leq \alpha V_{\infty}^{\mu_N}(x(n)) \leq V_N(x(n))\leq V_\infty(x(n))
	\end{align}
	holds for all $n \in \N_0$.
\end{proposition}
Note that Proposition \eqref{Preliminaries:prop:trajectory a posteriori estimate} is an {\em a posteriori} estimate since $V_N(x(n + 1))$ is not available at time $n$. A more conservative \textit{a priori} estimate is given in \cite[Theorem 7]{GP2009} using the following assumptions:
\begin{assumption}\label{Preliminaries:ass:apriori2}
	For given $N$, $\hat{N} \in \N$, $N \geq \hat{N} \geq 2$, there exists a constant $\gamma > 0$ such that for the open loop solution $x_{u_N}(k, x(n))$ given by \eqref{Setup:open loop solution} the inequalities
	\begin{align*}
		\frac{V_{\hat{N}}(x_{u_N}(N - \hat{N}, x(n)))}{\gamma + 1} & \leq \max_{j = 2, \ldots, \hat{N}} l(x_{u_N}(N - j, x(n)), \mu_{j - 1}(x_{u_N}(N - j, x(n)))) \\
		\frac{V_k(x_{u_N}(N - k, x(n)))}{\gamma + 1} & \leq l(x_{u_N}(N - k, x(n)), \mu_k(x_{u_N}(N - k, x(n))))
	\end{align*}
	hold for all $k \in \{\hat{N} + 1, \ldots, N\}$ and all $n \in \N_0$.
\end{assumption}
\begin{theorem}\label{Preliminaries:thm:apriori variante2}
	Consider $\gamma > 0$ and $N$, $\hat{N} \in \N$, $N \geq \hat{N}$ such that $(\gamma + 1)^{N - \hat{N}} > \gamma^{N - \hat{N} + 2}$ holds. If Assumption \ref{Preliminaries:ass:apriori2} is fulfilled for these $\gamma$, $N$ and $\hat{N}$, then \eqref{Preliminaries:prop:trajectory a posteriori estimate:eq2} holds for all $n \in \N_0$ where
	\begin{align}
		\label{Preliminaries:thm:apriori variante2:eq1}
		\alpha := \frac{(\gamma + 1)^{N - \hat{N}} - \gamma^{N - \hat{N} + 2}}{(\gamma + 1)^{N - \hat{N}}}.
	\end{align}
\end{theorem}
Comparing the estimates from \cite{GP2009, GR2008, SX1997}, we call the maximal value of $\alpha$ satisfying \eqref{Preliminaries:prop:trajectory a posteriori estimate:eq1} {\it local suboptimality degree} if $x(n) \in \X$ is fixed. For a given closed loop trajectory $x(\cdot)$ we call $\alpha := \max \{ \alpha \mid \text{\eqref{Preliminaries:prop:trajectory a posteriori estimate:eq1} holds $\forall n \in \N_0$} \}$ the {\it closed loop suboptimality degree} and for a given set $\X$ we call $\alpha := \max \{ \alpha \mid \text{\eqref{Preliminaries:prop:trajectory a posteriori estimate:eq1} with $x(n) = x$ holds $\forall x \in \X$} \}$ the {\it global suboptimality degree}.

Unfortunately, we cannot expect the relaxed Lyapunov inequality \eqref{Preliminaries:prop:trajectory a posteriori estimate:eq1} or Assumption \ref{Preliminaries:ass:apriori2} to hold in practice for the following reason: In many cases the discrete time system \eqref{Setup:nonlinear discrete time system} is obtained from a discretization of a continuous time system, e.g. sampling with zero order hold, see \cite{NT2004,NTK1999}. Hence, even if, e.g., a continuous time system is stabilizable to a setpoint $x^*$ and no numerical errors occur during optimization and integration, the corresponding sampled--data system is most likely practically stabilizable at $x^*$ only.

For this reason, the a posteriori and a priori estimates from Proposition \ref{Preliminaries:prop:trajectory a posteriori estimate} and Theorem \ref{Preliminaries:thm:apriori variante2} have been extended to cover the case of practical stability as well, see also \cite[Proposition 14 and Theorem 20]{GP2009}.  Since these practical suboptimality estimates can be used in a similar manner as in the non--practical case, we show stability of the horizon adaptation technique and adaptation strategies for the non--practical case only. Corresponding results can be found in \cite[Chapter 3 and 4]{P2009b}.

\section{Adapting the NMPC Scheme}
\label{Section:Adapting the NMPC Scheme}

Since the suboptimality estimates from Proposition \ref{Preliminaries:prop:trajectory a posteriori estimate} and Theorem \ref{Preliminaries:thm:apriori variante2} are computable online, we may utilize them to repeatedly adapt the optimization horizon. To this end, we extend our notation of the horizon length from $N$ to $N_n$, of the corresponding local suboptimality degree from $\alpha$ to $\alpha(N_n)$ and of the closed loop control law from $\mu_N(\cdot)$ to $\mu_{(N_n)}(\cdot)$ where $n$ indicates the discrete time instant. Defining a fixed suboptimality bound $\overline{\alpha} \in (0, 1)$, we propose the following algorithm to guarantee local suboptimality degree $\overline{\alpha}$:

\begin{algorithm}\label{ANMPC:alg:algorithm}
	Set $n := 0$ and choose $\overline{\alpha} \in (0, 1)$ and $N_n \in \N$.
	\begin{enumerate}
		\item Obtain new measurements $x(n)$.
		\item Set $\tilde{\alpha} = 0$. While $\tilde{\alpha} \leq \overline{\alpha}$ do
		\begin{enumerate}
			\item Compute the open loop optimal control sequence $u_N(\cdot, x(n))$ from \eqref{Setup:open loop control}
			\item Compute suboptimality degree $\tilde{\alpha} := \alpha(N_n)$ from, e.g., Proposition \ref{Preliminaries:prop:trajectory a posteriori estimate} or Theorem \ref{Preliminaries:thm:apriori variante2}
			\item If $\tilde{\alpha} \geq \overline{\alpha}$: Call shortening strategy for $N_n$ \\
			Else: Call prolongation strategy for $N_n$
		\end{enumerate}
		\item Implement $\mu_{(N_n)}(x(n)) := u(0, x(n))$, set $n := n + 1$ and goto Step 1.
	\end{enumerate}
\end{algorithm}

The problem which we are facing for such an adaptive MPC algorithm is the fact that none of the existing stability proofs, see, e.g., \cite{GMTT2005, GP2009, GPSW2010}, can be applied in this context since these results assume $N$ to be constant while here the optimization horizon $N_n$ may change in every step of the MPC algorithm. The major obstacle to apply the idea of Proposition \ref{Preliminaries:prop:trajectory a posteriori estimate} in the context of varying optimization horizons $N_n$ is the lack of a common Lyapunov function along the closed loop. To compensate for this deficiency, we assume that if for a horizon length $N_n$ we have $\alpha(N_n) \geq \overline{\alpha}$, then the controller shows a bounded guaranteed performance if $N_n$ is increased. For ease of notation, we give this assumption in a set valued manner, however, within the following stability proof it is only required to hold along the closed loop.

\begin{assumption}\label{ANMPC:ass:enhanced stabilizing}
	Given an initial value $x \in \X$ and a horizon length $N < \infty$ such that $\mu_N(\cdot)$ guarantees local suboptimality degree $\alpha(N) \geq \overline{\alpha}$, $\overline{\alpha} \in (0, 1)$, we assume that for $\widetilde{N} \geq N$, $\widetilde{N} < \infty$, there exist constants $C_l, C_\alpha > 0$ such that the inequalities
	\begin{align}
		\label{ANMPC:ass:enhanced stabilizing:eq1}
		l(x, \mu_{N}(x)) & \leq C_l l(x, \mu_{\widetilde{N}}(x)) \frac{V_{\widetilde{N}}(x) - V_{\widetilde{N}}(f(x, \mu_N(x))}{V_{\widetilde{N}}(x) - V_{\widetilde{N}}(f(x, \mu_{\widetilde{N}}(x))} \\
		\label{ANMPC:ass:enhanced stabilizing:eq2}
		\alpha(N) & \leq \frac{1}{C_\alpha} \alpha(\widetilde{N})
	\end{align}
	hold where $\alpha(\widetilde{N})$ is the local suboptimality degree of the controller $\mu_{\widetilde{N}}(\cdot)$ corresponding to the horizon length $\widetilde{N}$.
\end{assumption}

Note that Assumption \ref{ANMPC:ass:enhanced stabilizing} is indeed not very restrictive since we allow for non--monotone developments of the suboptimality degree $\alpha(\cdot)$ if the horizon length is increased which may occur as shown in \cite{DPM2007}. Moreover, we only make sure that if a certain suboptimality degree $\overline{\alpha} \in (0, 1)$ holds for a horizon length $N$, then the estimate $\alpha(\widetilde{N})$ does not drop below zero if the horizon length $\widetilde{N}$ is increased.

Using Assumption \ref{ANMPC:ass:enhanced stabilizing} to hold along the closed loop, we obtain stability and a performance estimate of the closed loop for changing horizon lengths:

\begin{theorem}\label{ANMPC:thm:stability of adaptive NMPC}
	Consider $\overline{\alpha} \in (0, 1)$ and a sequence $(N_n)_{n \in \N_0}$, $N_i \in \N$, where $N^\star = \max \{ N_n \, | \; n \in \N_0 \}$, such that the NMPC feedback law $\mu_{(N_n)}$ defining the closed loop solution \eqref{Setup:closed loop solution} guarantees
	\begin{align}
		\label{ANMPC:thm:stability of adaptive NMPC:eq1}
		V_{N_n}(x(n)) \geq V_{N_n}(x(n+1)) + \overline{\alpha} l(x(n), \mu_{N_n}(x(n)))
	\end{align}
	for all $n \in \N_0$. Moreover suppose Assumption \ref{ANMPC:ass:enhanced stabilizing} to hold for all pairs $(x(n), N_n)$, $n \in \N_0$. Then we obtain
	\begin{align}
		\label{ANMPC:thm:stability of adaptive NMPC:eq2}
		\alpha_{C} V_\infty(x(n)) \leq \alpha_{C} V_\infty^{\mu_{(N_n)}}(x(n)) \leq V_{N^\star}(x(n)) \leq V_\infty(x(n))
	\end{align}
	to hold for all $n \in \N_0$ with $\alpha_{C} := \min\limits_{j \in \N_0, j \geq n} \frac{C_\alpha(j)}{C_l(j)} \overline{\alpha}$ and $C_\alpha(j), C_l(j)$ from \eqref{ANMPC:ass:enhanced stabilizing:eq1}, \eqref{ANMPC:ass:enhanced stabilizing:eq2} for $x = x(j)$, $j \geq n \in \N_0$.
\end{theorem}

\begin{proof}
	Given a pair $(x(n), N_n)$, Assumption \ref{ANMPC:ass:enhanced stabilizing} guarantees $\alpha(N_n) \leq \alpha(\widetilde{N})/ C_\alpha^{(n)}$ for $\widetilde{N} \geq N_n$. Now we choose $\widetilde{N} = N^\star$ within this local suboptimality estimate. Hence, we obtain $\overline{\alpha} \leq \alpha(N_n) \leq \alpha(N^\star) / C_\alpha(n)$ using the relaxed Lyapunov inequality \eqref{ANMPC:thm:stability of adaptive NMPC:eq1}. Multiplying by $l(x(n), \mu_{N_n}(x(n)))$ and using \eqref{ANMPC:ass:enhanced stabilizing:eq1}, we can conclude $\overline{\alpha} l(x(n), \mu_{N_n}(x(n))) \leq \frac{C_l(n)}{C_\alpha(n)} \left( V_{N^\star}(x(n)) - V_{N^\star}(x(n+1)) \right)$. Since the latter condition relates the closed loop varying optimization horizon to a fixed one, it allows us to use an identical telescope sum argument as in the proof of \cite[Proposition 3]{GP2009}. Hence, summing the running costs along the closed loop trajectory reveals $\alpha_{C} \sum_{j = n}^{K} l(x(j), \mu_{N_j}(x(j))) \leq V_{N^\star}(x(n)) - V_{N^\star}(x(K+1))$ where we defined $\alpha_{C} := \min\limits_{j \in [n, \ldots, K]} \frac{C_\alpha(j)}{C_l(j)} \overline{\alpha}$ with constants $C_\alpha(j)$ and $C_l(j)$ from \eqref{ANMPC:ass:enhanced stabilizing:eq1} and \eqref{ANMPC:ass:enhanced stabilizing:eq2} for $x = x(j)$ and $j \in \{n, \ldots, K\}$. Since $V_{N^\star}(x(K+1)) \geq 0$ holds, taking $K$ to infinity reveals
	\begin{align*}
		\alpha_{C} V_\infty^{\mu_{(N_n)}}(x(n)) = \alpha_C \lim\limits_{K \rightarrow \infty} \sum\limits_{j = n}^{K} l(x(j), \mu_{N_j}(x(j))) \leq V_{N^\star}(x(n)).
	\end{align*}
	Since the first and the last inequality of \eqref{ANMPC:thm:stability of adaptive NMPC:eq2} hold by the principle of optimality, the assertion follows.
\end{proof}

Similar to Proposition \ref{Preliminaries:prop:trajectory a posteriori estimate}, Theorem \ref{ANMPC:thm:stability of adaptive NMPC} can be extended to the practical case, cf. \cite[Chapter 4]{P2009b}. We like to point out that Theorem \ref{ANMPC:thm:stability of adaptive NMPC} is a generalization of Proposition \ref{Preliminaries:prop:trajectory a posteriori estimate} which is reobtained if $N_n = N$ for all $n \in \N_0$. 

Note that the closed loop estimate $\alpha_C$ in \eqref{ANMPC:thm:stability of adaptive NMPC:eq2} may be smaller than the local suboptimality bound $\overline{\alpha}$.  In particular, since $l(x, \mu_{\widetilde{N}}(x))$ may tend to zero if $\widetilde{N}$ is increased, we obtain that $C_l$ in \eqref{ANMPC:ass:enhanced stabilizing:eq1} is in general unbounded. The special case $l(x, \mu_{\widetilde{N}}(x)) = 0$, however, states that the equilibrium of our problem has been reached and can be neglected in this context, i.e. outside the equilibrium $\alpha_C > 0$ is always retained. Yet, $\alpha_C$ may become very small depending on $C_\alpha$ and $C_l$ from Assumption \ref{ANMPC:ass:enhanced stabilizing}. During our numerical experiments, however, no such case occured, see also Section \ref{Section:Numerical Results}.


\section{Adaptation Strategies}
\label{Section:Adaptation Strategies}

As we have seen in Theorem \ref{ANMPC:thm:stability of adaptive NMPC}, the methods from Proposition \ref{Preliminaries:prop:trajectory a posteriori estimate} and Theorem \ref{Preliminaries:thm:apriori variante2} can be applied to compute a suboptimality estimate for a given pair $(x(n), N_n)$. Yet, these local estimates have to reinterpreted along the closed loop. In particular, we require the existence of a finite horizon length $N_n$ guaranteeing stability with suboptimality degree greater than $\overline{\alpha}$ in order to conclude finite termination of the algorithm proposed in Section \ref{Section:Adapting the NMPC Scheme}.

\begin{assumption}\label{ANMPC:ass:stabilizable}
	Given $\overline{\alpha} \in (0, 1)$, for all $x_0 \in \X$ there exists a finite horizon length $\overline{N} = N(x_0) \in \N$ such that the relaxed Lyapunov inequality \eqref{Preliminaries:prop:trajectory a posteriori estimate:eq1} holds with $\alpha(N) \geq \overline{\alpha}$ for all horizon lengths $N \geq \overline{N}$.
\end{assumption}

Note that Assumption \ref{ANMPC:ass:stabilizable} is satisfied if $\overline{\alpha}$ is small enough and $\overline{N}$ is large enough, see, e.g., \cite{GMTT2005}.

\subsection{Simple Adaptation Strategies}

A basic adaptation technique for the horizon length can be obtained using Proposition \ref{Preliminaries:prop:trajectory a posteriori estimate}. In particular, we can repeatedly shorten the horizon length and check Assumption \eqref{Preliminaries:prop:trajectory a posteriori estimate:eq1} as the solution evolves:

\begin{theorem}\label{ANMPC:thm:stepsize shortening}
	Consider an optimal control problem \eqref{Setup:open loop control}, \eqref{Setup:open loop solution} with initial value $x_0 = x(n)$, $N_n \in \N$ and $\overline{\alpha} \in (0, 1)$ to be fixed and denote the optimal control sequence by $u^\star$. Suppose there exists an integer $\overline{k} \in \N_0$, $0 \leq \overline{k} < N_n$ such that
	\begin{align}
		\label{ANMPC:thm:stepsize shortening:eq1}
		V_{N_n-k}(x_{u_N}(k, x_0)) \geq V_{N_n-k}(x_{u^\star}(k + 1, x_0)) + \overline{\alpha} l(x_{u^\star}(k, x_0), \mu_{N_n-k}(x_{u^\star}(k, x_0)))
	\end{align}
	holds true for all $0 \leq k \leq \overline{k}$. Then, setting $N_{n + k}= N_n - k$ and $\mu_{N_{n + k}}(x(n + k))= u^\star(k)$ for $0 \leq k \leq \overline{k}$, inequality \eqref{ANMPC:thm:stability of adaptive NMPC:eq1} holds for $k = n, \ldots, n + \overline{k}$ with $\alpha = \overline{\alpha}$.
	\end{theorem}
\begin{proof}
	The assertions follows directly from the fact that for $\mu_{N_{n + k}}(x(n + k))= u^\star(k)$ the closed loop satisfies $x(n + k) = x_{u^\star}(k, x(n))$. Hence, \eqref{ANMPC:thm:stability of adaptive NMPC:eq1} follows from \eqref{ANMPC:thm:stepsize shortening:eq1}. 
\end{proof}

Note that the result of Theorem \ref{ANMPC:thm:stepsize shortening} can be extended to consider an $m$--step feedback as defined in \cite{GPSW2010} by supposing $\overline{k} \geq m - 1$. 
With the choice $N_{n+k}=N_n-k$, due to the principle of optimality we obtain that the optimal control problems within the next $\overline{k}$ NMPC iterations are already solved since $\mu_{N_n-k}(x(n+k))$ can be obtained from the optimal control sequence $u^\star(\cdot)$ computed at time $n$. This implies that the most efficient way for the reducing strategy is not to reduce $N_n$ itself but rather to reduce the horizons $N_{n + k}$ by $k$ for the subsequent sampling instants $n + 1, \ldots, n + \overline{k}$, i.e., we choose the initial guess of the horizon $N_{n+1} = N_n-1$. Still, if the a posteriori estimate is used, the evaluation of \eqref{ANMPC:thm:stepsize shortening:eq1} requires the solution of an additional optimal control problem in each step.

In order to to use the \textit{a priori} estimate given by Theorem \ref{Preliminaries:thm:apriori variante2} the following result can be used as a shortening strategy:

\begin{theorem}\label{ANMPC:thm:stepsize shortening2}
	Consider an optimal control problem \eqref{Setup:open loop control}, \eqref{Setup:open loop solution} with initial value $x_0 = x(n)$ and $N_n, \hat{N} \in \N$, $N_n \geq \hat{N} \geq 2$ and denote the optimal control sequence by $u^\star$. Moreover, $\overline{\alpha} \in (0, 1)$ is supposed to be fixed inducing some $\overline{\gamma}(\cdot)$ via \eqref{Preliminaries:thm:apriori variante2:eq1}. If there exists an integer $\overline{k} \in \N_0$, $0 \leq \overline{k} < N_n - \hat{N} - 1$ such that for all $0 \leq k \leq \overline{k}$ there exist $\gamma_n > 0$, $\gamma_{n} < \overline{\gamma}(N_n - k)$ satisfying
	\begin{align}
		\label{ANMPC:thm:stepsize shortening2:eq1}
		V_{\hat{N}}(x_{u^\star}(N_n - \hat{N}, x_0)) & \leq (\gamma_{n} + 1) \max_{j = 2, \ldots, \hat{N}} l(x_{u^\star}(N_n - j, x_0), \mu_{j - 1}(x_{u^\star}(N_n - j, x_0))) \\
		\label{ANMPC:thm:stepsize shortening2:eq2}
		V_{j_k}(x_{u^\star}(N_n- j_k, x_0)) & \leq (\gamma_{n} + 1) l(x_{u^\star}(N_n - j_k, x_0), \mu_{j_k}(x_{u^\star}(N_n - j_k, x_0)))
	\end{align}
	for all $j_k \in \{\hat{N} + 1, \ldots, N_n - k\}$. Then, setting $N_{n + k}= N_n - k$ and $\mu_{N_{n + k}}(x(n + k))= u^\star(k)$ for $0 \leq k \leq \overline{k}$, inequality \eqref{ANMPC:thm:stability of adaptive NMPC:eq1} holds for $k = n, \ldots, n + \overline{k}$ with $\alpha = \overline{\alpha}$.
\end{theorem}
\begin{proof}
	Since \eqref{ANMPC:thm:stepsize shortening2:eq1}, \eqref{ANMPC:thm:stepsize shortening2:eq2} hold for $k = 0$ with $\gamma_n > 0$ and $\gamma_n < \overline{\gamma}(N_n)$, Theorem \ref{Preliminaries:thm:apriori variante2} guarantees that the local suboptimality degree is at least as large as $\overline{\alpha}$. If $\overline{k} > 0$ holds, we can make use of the fact that for $\mu_{N_{n + k}}(x(n + k))= u^\star(k)$ the closed loop satisfies $x(n + k) = x_{u^\star}(k, x(n))$. By \eqref{ANMPC:thm:stepsize shortening2:eq1}, \eqref{ANMPC:thm:stepsize shortening2:eq2}, we obtain Assumption \ref{ANMPC:ass:enhanced stabilizing} to hold along the closed loop. Accordingly, the assertion follows from Theorem \ref{Preliminaries:thm:apriori variante2} which concludes the proof.
\end{proof}

Similar the a posteriori case, the shortening strategy given by Theorem \ref{ANMPC:thm:stepsize shortening2} can be extended to consider an $m$--step feedback as defined in \cite{GPSW2010} by supposing $\overline{k} \geq m - 1$. Note that while the a priori estimate from Theorem \ref{Preliminaries:thm:apriori variante2} is slightly more conservative than the result from Proposition \ref{Preliminaries:prop:trajectory a posteriori estimate}, it is also computationally less demanding if the value $\hat{N}$ is small.

The shortening strategies induced by Theorems \ref{ANMPC:thm:stepsize shortening} and \ref{ANMPC:thm:stepsize shortening2} can be extended to cover the case of practical stability. To this end, the inequalities \eqref{ANMPC:thm:stepsize shortening:eq1}, \eqref{ANMPC:thm:stepsize shortening2:eq1} and \eqref{ANMPC:thm:stepsize shortening2:eq2} have to be replaced by their practical equivalents given in \cite[Proposition 14]{GP2009} and \cite[Theorem 21]{GP2009}.

In contrast to these efficient and simple shortening strategies it is quite difficult to obtain efficient methods for prolongating the optimization horizon $N_n$. In order to provide a simple prolongating strategy, we invert the approach of Theorem \ref{ANMPC:thm:stepsize shortening}:

\begin{theorem}\label{ANMPC:thm:stepsize prolongation}
	Consider an optimal control problem \eqref{Setup:open loop control}, \eqref{Setup:open loop solution} with initial value $x_0 = x(n)$ and $N_n \in \N$. Moreover, for fixed $\overline{\alpha} \in (0, 1)$ supposed Assumption \ref{ANMPC:ass:stabilizable} to hold. Then, any algorithm which iteratively increases the optimization horizon $N_n$ terminates in finite time and computes a horizon length $N_n$ such that \eqref{ANMPC:thm:stability of adaptive NMPC:eq1} holds with local suboptimality degree $\overline{\alpha}$.
\end{theorem}
\begin{proof}
	Follows directly from Assumption \ref{ANMPC:ass:stabilizable}.
\end{proof}

Note that the prolongation strategy described in Theorem \ref{ANMPC:thm:stepsize prolongation} only requires Assumption \ref{ANMPC:ass:stabilizable} to hold. This allows us to use any of the suboptimality estimates stated in Section \ref{Section:Setup and Preliminaries}. Unfortunately, if \eqref{ANMPC:thm:stability of adaptive NMPC:eq1} does not hold, it is in general difficult to assess by how much $N_n$ should be increased such that \eqref{ANMPC:thm:stability of adaptive NMPC:eq1} holds for the increased $N_n$. The most simple strategy of increasing $N_n$ by one in each iteration shows satisfactory results in practice, however, when starting the iteration with $N_n$, in the worst case \eqref{ANMPC:thm:stability of adaptive NMPC:eq1} has to be checked $\overline{N} - N_n + 1$ times at each sampling instant. In contrast to the shortening strategy, the principle of optimality cannot be used here to establish a relation between the optimal control problems for different $N_n$ and, moreover, these problems may exhibit different solution structures which makes it a hard task to provide a suitable initial guess for the optimization algorithm. 

\subsection{Advanced Adaptation Strategies}

Since the shortening strategies based on both the a posteriori and the a priori estimates can be implemented with negligible additional computational effort, we focus on the prolongation of the horizon. To reduce the additional effort, we analyze the relationsship of $\alpha(N_n)$ and $\gamma(N_n)$ given by Theorem \ref{Preliminaries:thm:apriori variante2}. If we consider $\overline{\alpha} \in (0, 1)$, we obtain a lower bound for $N_n$ from \eqref{Preliminaries:thm:apriori variante2:eq1} by
\begin{align}
	\label{ANMPC:eq:N-alpha-gamma relation}
	N_n \geq \left\lceil \hat{N} + \frac{2 \ln(\gamma(N_n)) - \ln(1 - \overline{\alpha})}{\ln(\gamma(N_n) + 1) - \ln(\gamma(N_n))} \right\rceil =: \Phi(N_n)
\end{align}
for fixed $x(n)$, $\hat{N}$ and $\overline{\alpha}$. Since we want to guarantee local suboptimality degree $\overline{\alpha}$ and $N_n$ to be as small as possible, we seek a horizon length $N_n$ satisfying $N_n = \Phi(N_n)$, i.e. a fixed point of the function $\Phi(\cdot)$.


\begin{theorem}\label{ANMPC:thm:fixed point}
	Consider $N_n, \hat{N} \in \N$, $N_n \geq \hat{N} \geq 2$, and $\overline{\alpha} \in (0, 1)$ to be fixed and $\gamma(N_n)$ to minimally satisfy Assumption \ref{Preliminaries:ass:apriori2}. If for a given $n \in \N_0$ there exists a constant $\theta \in [0, 1)$ such that the function $\Phi(\cdot)$ defined in \eqref{ANMPC:eq:N-alpha-gamma relation} satisfies
	\begin{align}
		\label{ANMPC:thm:fixed point:eq2}
		| \Phi(\Phi(N_n)) - \Phi(N_n) | \leq \theta | \Phi(N_n) - N_n | \qquad \theta \in [0, 1) \; \forall N_n \geq \hat{N}.
	\end{align}
	and $\Phi^k(N_n) \geq \hat{N}$ for all $k \in \N$, then there exists $N_n^\star \in \N$ with $N_n^\star = \Phi(N_n^\star)$ and $\Phi^k(N_n) \rightarrow N_n^\star$, $k \rightarrow \infty$. If we use $N_n = N_n^\star$ in Algorithm \ref{ANMPC:alg:algorithm} for all $n \in \N_0$ and Assumption \ref{ANMPC:ass:enhanced stabilizing} holds for $x = x(n)$ and all $n \in \N_0$, then the closed loop solution \eqref{Setup:closed loop solution} is asymptotically stable and exhibits local suboptimality degree $\alpha(N_n^\star) \geq \overline{\alpha}$.
\end{theorem}
\begin{proof}
	Since $\gamma(N_n)$ satisfies all requirements of Theorem \ref{Preliminaries:thm:apriori variante2}, we can obtain an estimate $\alpha(N_n)$ via \eqref{Preliminaries:thm:apriori variante2:eq1}. In order to guarantee a certain degree of suboptimality $\overline{\alpha}$ we have to show $\overline{\alpha} \leq \alpha(N_n) = \frac{(\gamma + 1)^{N_n - \hat{N}} - \gamma^{N_n - \hat{N} + 2}}{(\gamma + 1)^{N_n - \hat{N}}}$. This can be solved for $N_n$ giving $N_n \geq \Phi(N_n)$ with $\Phi(\cdot)$ from \eqref{ANMPC:eq:N-alpha-gamma relation}. Due to \eqref{ANMPC:thm:fixed point:eq2} we have
	\begin{align}
		\label{ANMPC:thm:fixed point:proof:eq1}
		| \Phi^{k}(N_n) - \Phi^{k-1}(N_n) | \leq \theta^{k-1} | \Phi(N_n) - N_n |.
	\end{align}
	Since $\theta \in [0, 1)$ the right hand side of \eqref{ANMPC:thm:fixed point:proof:eq1} tends to zero. Hence, there exists an index $\overline{k} \in \N$ such that $\theta^{\overline{k} - 1} | \Phi(N_n) - N_n | < 1$. Defining the sequence of optimization horizons via $(N_n^{(i)})_{i \in \N_0} := (\Phi^{i}(N_n))_{i \in \N_0}$ we obtain $N_n^{(j)} = N_n^{(k)} \geq \hat{N}$ for all $j, k \geq \overline{k}$. Hence, $(N_n^{(i)})_{i \in \N_0}$ is converging and $N_n^\star = \Phi(N_n^\star)$ holds for $N_n^\star = N_n^{(\overline{k})}$.\\
	Choosing $N_n = N_n^\star$, the local suboptimality degree satisfies $\alpha(N_n) \geq \overline{\alpha}$ by construction of $\Phi(\cdot)$. Hence, a new initial value can be obtained by implementing the controller in a receding horizon fashion. Since we can apply this procedure along the resulting trajectory, i.e. for all $n \in \N_0$, asymptotic stability of the closed loop solution \eqref{Setup:closed loop solution} follows by Assumption \ref{ANMPC:ass:enhanced stabilizing} and Theorem \ref{ANMPC:thm:stability of adaptive NMPC} and $\alpha(N_n^\star) \geq \overline{\alpha}$ follows directly from Theorem \ref{Preliminaries:thm:apriori variante2}.
\end{proof}

Note that, in general, we cannot a priori check whether $\Phi(\cdot)$ satisfies \eqref{ANMPC:thm:fixed point:eq2}. Moreover, an algorithm derived from Theorem \ref{ANMPC:thm:fixed point} may cause overshoots. Numerical experience has shown that $\sigma = 5$ is a suitable choice to bound the change of the the horizon length, yet, this variable should be chosen with respect to the considered problem. Additionally, numerical simulations indicate that the ``best'' choice of $\sigma$ depends on the occuring horizon lengths $N_n$, i.e. larger horizons allow for larger choices of $\sigma$.

Theorem \ref{ANMPC:thm:fixed point} can also be utilized to shorten the horizon. However, the computation of $N^\star$ requires nonnegliable effort. Hence, this strategy should only be considered when $\alpha(N) < \overline{\alpha}$. Yet, $N_{i + 1} := \Phi(N_{i})$ may be a suitable choice for the optimization horizon in the subsequent optimal control problem.

Different from the fixed point idea, a map $\Psi(\cdot)$ can be designed which generates a sequence of horizons $(N_n^{(i)})$ via $N_n^{(i + 1)} := \Psi(N_n^{(i)})$ such that the suboptimality estimate $\alpha(N_n^{(i)})$ is monotonely increasing:

\begin{lemma}\label{ANMPC:lem:monotonicity alpha}
	Suppose $N_n, \hat{N} \in \N$, $N_n \geq \hat{N} \geq 2$, $\overline{\alpha} \in (0, 1)$ and $0 \leq \delta < 1 - \alpha(N_n)$ are given and Assumption \ref{Preliminaries:ass:apriori2} holds. Suppose there exists a constant $\vartheta > 0$ such that $\gamma(\tilde{N}) \leq \vartheta \gamma(N_n)$ holds for
	\begin{align}
		\label{ANMPC:lem:monotonicity alpha:eq1}
		\tilde{N} & \geq \left\lceil \hat{N} + \left( \frac{ \ln \left( \left( \frac{\gamma(N_n)}{\gamma(N_n) + 1}\right)^{N_n - \hat{N}} - \frac{\delta}{\gamma(N_n)^2} \right) - 2 \ln (\vartheta) }{ \ln \left( \vartheta \gamma(N_n) \right) - \ln \left( \vartheta \gamma(N_n) + 1 \right) } \right) \right\rceil =: \Psi(N_n),
	\end{align}
	then $\alpha(\tilde{N})$ as defined in \eqref{Preliminaries:thm:apriori variante2:eq1} with $\gamma = \gamma(\tilde{N})$ from Assumption \ref{Preliminaries:ass:apriori2} satisfies
	\begin{align}
		\label{ANMPC:lem:monotonicity alpha:eq2}
		\alpha(\tilde{N}) \geq \alpha(N_n) + \delta
	\end{align}
\end{lemma}
\begin{proof}
	In order to show $\alpha(\tilde{N}) \geq \alpha(N) + \delta$ we use \eqref{Preliminaries:thm:apriori variante2:eq1} in \eqref{ANMPC:lem:monotonicity alpha:eq2} which gives us
	\begin{align*}
		\gamma(\tilde{N})^2 \left( \frac{\gamma(\tilde{N})}{\gamma(\tilde{N}) + 1} \right)^{\tilde{N} - \hat{N}} \leq \gamma(N)^2 \left( \frac{\gamma(N)}{\gamma(N) + 1}\right)^{N - \hat{N}} - \delta.
	\end{align*}
	Overestimating the left hand side using $\gamma(\tilde{N}) \leq \vartheta \gamma(N)$ this leaves us to show
	\begin{align*}
		\vartheta^2 \gamma(N)^2 \left( \frac{\vartheta \gamma(N)}{\vartheta \gamma(N) + 1} \right)^{\tilde{N} - \hat{N}} \leq \gamma(N)^2 \left( \frac{\gamma(N)}{\gamma(N) + 1}\right)^{N - \hat{N}} - \delta
	\end{align*}
	to guarantee \eqref{ANMPC:lem:monotonicity alpha:eq2}. Since $\vartheta > 0$, this inequality is equivalent to
	\begin{align*}
		\left( \tilde{N} - \hat{N} \right) \left[ \ln \left( \frac{ \vartheta \gamma(N) }{ \vartheta \gamma(N) + 1 } \right) \right] \leq \ln \left( \left( \frac{\gamma(N)}{\gamma(N) + 1}\right)^{N - \hat{N}} - \frac{\delta}{\gamma(N)^2} \right) - 2 \ln (\vartheta)
	\end{align*}
	Using negative definiteness of $\ln \left( \frac{ \vartheta \gamma(N) }{ \vartheta \gamma(N) + 1 } \right)$ and \eqref{ANMPC:lem:monotonicity alpha:eq1} the assertion follows.
\end{proof}

Similar to $\Phi(\cdot)$ from Theorem \ref{ANMPC:thm:fixed point} the map $\Psi((\cdot)$ may be used to shorten the horizon, a feature which can be avoided easily:

\begin{lemma}\label{ANMPC:lem:monotonicity phi}
	If $\vartheta \geq 1$ and $\delta \geq 0$ hold, then we have $\Psi(N) > N$ for $\Psi(\cdot)$ from \eqref{ANMPC:lem:monotonicity alpha:eq1}.
\end{lemma}
\begin{proof}
	In the special case $\delta = 0$, \eqref{ANMPC:lem:monotonicity alpha:eq1} simplifies to
	\begin{align*}
		\Psi(N) := \left\lceil \hat{N} + \frac{ \left( N - \hat{N} \right) \left( \ln \left( \gamma(N) + 1 \right) - \ln \left( \gamma(N) \right) \right) + 2 \ln \left( \vartheta \right) }{ \ln \left( \vartheta \gamma(N) + 1 \right) - \ln \left( \vartheta \gamma(N) \right) } \right\rceil.
	\end{align*}
	Since $\frac{\ln(x + 1) - \ln(x)}{\ln(\vartheta(x + 1)) - \ln(\vartheta(x))} > 1$ and $\vartheta \geq 1$ hold, we obtain $\Psi(N) > N$. Moreover, increasing $\delta$ results in an enlarged value $\Psi(N)$ showing the assertion.
\end{proof}

Obtaining a suitable approximation of $\vartheta$ is the most crucial part of the monotone prolongation method. In order to avoid computing the open loop optimal control \eqref{Setup:open loop control} for all initial values $x \in \X$ and all $N \geq \hat{N}$ and deriving the corresponding values $\gamma(\cdot)$ to obtain $\vartheta$, one can iteratively update the value of $\vartheta$ by setting $\vartheta := \max\left\{ \vartheta, \gamma(N_n^{(i+1)}) / \gamma(N_n^{(i)}) \right\}$. This method is not only computationally cheap and gives us a lower bound for $\vartheta$, it also moderates a possible overshoot. Note that this approximation has to be restarted for each $n \in \N_0$.

Next, we use Lemma \ref{ANMPC:lem:monotonicity alpha} to show that a prolongation strategy based on $\Psi(\cdot)$ in Step 2 of Algorithm \ref{ANMPC:alg:algorithm} terminates in finite time:

\begin{theorem}\label{ANMPC:thm:monotone iteration}
	Suppose Assumptions \ref{Preliminaries:ass:apriori2} and \ref{ANMPC:ass:stabilizable} hold and suppose $\vartheta \geq 1$. Then a finite number of iterations $N_n^{(i + 1)} := \Psi(N^{(i)})$ with $\Psi(\cdot)$ according to Lemma \ref{ANMPC:lem:monotonicity alpha} are required to obtain a horizon length $N_n^\star$ which guarantees local suboptimality degree $\alpha(N_n^\star) \geq \overline{\alpha}$. 
\end{theorem}
\begin{proof}
	Using the stopping criterion $\alpha(N_n) \geq \overline{\alpha}$ of Algorithm \ref{ANMPC:alg:algorithm} to define $\delta = \overline{\alpha} - \alpha(N_n)$, we always have $\delta > 0$. Hence, by Lemma \ref{ANMPC:lem:monotonicity phi} we can conclude that the horizon length $N_n^{(i)}$ is increasing in every step of the iteration due to $\vartheta \geq 1$. Since Assumption \ref{ANMPC:ass:stabilizable} guarantees the existence of a finite horizon length $\overline{N} \in \N$, $\overline{N} < \infty$, such that $\alpha(N_n) \geq \overline{\alpha}$ holds for all $N_n \geq \overline{N}$, the iteration $N_n^{(i + 1)} := \Psi(N_n^{(i)})$ terminates in finite time. Hence, choosing $N_n = N_n^\star$ we have that $\alpha(N_n) \geq \overline{\alpha}$ is guaranteed by the stopping criterion. 
\end{proof}

Note that we do not assume $\gamma(\cdot)$ in \eqref{ANMPC:eq:N-alpha-gamma relation} to be computed in a specific way but only to satisfy Assumption \ref{Preliminaries:ass:apriori2}. Hence, using Theorems \ref{ANMPC:thm:fixed point} and \ref{ANMPC:thm:monotone iteration} a suitable horizon $N_n$ can be obtained such that $\alpha(N_n) \geq \overline{\alpha}$ holds even if some $\tilde{\gamma}(\cdot) \geq \gamma(\cdot)$ are used. Therefore, also the a priori practical estimate \cite[Theorem 21]{GP2009} or the global estimate presented in \cite[Theorem 4.5 and 5.8]{GR2008} can be applied. Additionally, \eqref{Preliminaries:thm:apriori variante2:eq1} defines a bijective mapping relating $\alpha(N_n)$ and $\gamma(N_n)$ which allows us to use the a posteriori estimates of Proposition \ref{Preliminaries:prop:trajectory a posteriori estimate} and \cite[Proposition 14]{GP2009}:

\begin{lemma}\label{ANMPC:lem:bijectivity}
	Given $N_n, \hat{N} \in \N$, $N_n, \N_0 \geq 2$, the mapping $\Gamma: [0, \infty) \rightarrow (-\infty, 1]$, $\Gamma(x) := 1 - \frac{x^{N_n - \hat{N} + 2}}{(x + 1)^{N_n - \hat{N}}}$ is bijective.
\end{lemma}

\begin{proof}
	In order to show bijectivity, we use continuity of $\Gamma(\cdot)$ on $[0, \infty)$ and $\Gamma(0) = 1$ and $\lim_{x \rightarrow \infty} \Gamma(x) = - \infty$ to show surjectivity. In order to obtain injectivity, we show that $\Gamma(\cdot)$ is strictly monotone on $[0, \infty)$, i.e. for all $x \in [0, \infty)$ and $\varepsilon > 0$ we have $\Gamma(x) > \Gamma(x + \varepsilon)$. Using the definition of $\Gamma(\cdot)$ in this last inequality, we obtain $\left( 1 + \varepsilon/ (x + 1) \right)^{N_n - \hat{N}} < \left( 1 + \varepsilon / x \right)^{N_n - \hat{N} + 2}$ which holds true for $x > 0$ since $N_n - \hat{N} \geq 0$. Last, considering $x = 0$, we see that $\Gamma(\varepsilon) < 1$ for all $\varepsilon > 0$ and hence the assertion follows.
\end{proof}

Now, Lemma \ref{ANMPC:lem:bijectivity} allows us to solve \eqref{Preliminaries:thm:apriori variante2:eq1} for $\gamma(N_n)$. Since $\Gamma(\cdot)$ is twice continuously differentiable on $[0, \infty)$, this can be done effectively using Newton's method
\begin{align}
	\label{ANMPC:eq:gamma via Newton method}
	\gamma^{(k + 1)}(N_n) := \gamma^{(k)}(N_n) + \frac{1 - \alpha - \frac{\gamma^{(k)}(N_n)^{N_n - \hat{N} + 2}}{(\gamma^{(k)}(N_n) + 1)^{N_n - \hat{N}}}}{\left( \frac{\gamma^{(k)}(N_n)}{\gamma^{(k)}(N_n) + 1} \right)^{N_n - \hat{N} + 1} \left( N_n - \hat{N} + 2 + 2 \gamma^{(k)}(N_n) \right) }
\end{align}
where $\gamma^{(0)}(N_n) := 1$ can be used as the initial value due to strict monotonicity of $\Gamma(\cdot)$. However, since we do not expect $\gamma(N_n)$ to vary massively along the closed loop, the information from previous steps may be reused.

Using the a posteriori estimate, we obtain an additional degree of freedom since $\hat{N}$ can be chosen arbitrarily within the set $\{ 2, \ldots, N_n \}$. However, it is not clear which value is the best: For one, the smallest possible $\gamma(N_n) > 0$ is obtained by setting $\hat{N} := 2$. Formulas \eqref{ANMPC:eq:N-alpha-gamma relation} and \eqref{ANMPC:lem:monotonicity alpha:eq1}, however, logarithmically depend on $\gamma(N_n)$, i.e small values of $\gamma(N_n)$ might lead to overshoots.

\section{Numerical Results}
\label{Section:Numerical Results}

To illustrate the effectiveness of our adaptation strategies, we consider a digital redesign problem (cf.\ \cite{NG2006}) of an arm/\-rotor/\-platform (ARP) model stated in \cite{FK1996}:
\begin{align*}
	\dot{x}_{1}(t) & = x_{2}(t) + x_{6}(t) x_{3}(t) \\ \displaybreak[0]
	\dot{x}_{2}(t) & = -k_{1} x_{1}(t) / M - b_{1} x_{2}(t) / M + x_{6}(t) x_{4}(t) - m r b_{1} x_{6}(t) / M^{2} \\ \displaybreak[0]
	\dot{x}_{3}(t) & = - x_{6}(t) x_{1}(t) + x_{4}(t) \\ \displaybreak[0]
	\dot{x}_{4}(t) & = - x_{6}(t) x_{2}(t) - k_{1} x_{3}(t) / M - b_{1} x_{4}(t) / M + m r k_{1} / M^{2} \\ \displaybreak[0]
	\dot{x}_{5}(t) & = x_{6}(t) \\ \displaybreak[0]
	\dot{x}_{6}(t) & = -a_{1} x_{5}(t) - a_{2} x_{6}(t) + a_{1} x_{7}(t) + a_{3} x_{8}(t) - p_{1} x_{1}(t) - p_{2} x_{2}(t) \\ \displaybreak[0]
	\dot{x}_{7}(t) & = x_{8}(t) \\ \displaybreak[0]
	\dot{x}_{8}(t) & = a_{4} x_{5}(t) + a_{5} x_{6}(t) - a_{4} x_{7}(t) - (a_{5} + a_{6}) x_{8}(t) + u(t) / J
\end{align*}
For the digital redesign a continuous time full-state feedback $u_{0}$ was developed via backstepping such that the output $\zeta(t) := x_{5}(t) - \frac{a_{3}}{a_{1} - a_{2}a_{3}}[x_{6}(t) - a_{3} x_{7}(t)]$ is close to $x_{5}(t)$ and tracks a given reference signal $\zeta_{\text{ref}}(\cdot)$, see \cite[Chapter 7.3.2]{FK1996} for details on the backstepping design and the specification of the model parameters. The resulting continuous time solution of $x_{5,\text{ref}}(\cdot)$ is then used as a reference to compute a sampled--data control with zero order hold, cf. \cite{NG2006}. To solve the sequence of optimal control problems we use a direct approach and employ an SQP method to solve the resulting optimization problem. Here, we consider the initial value $x_0 = (0, 0, 0, 0, 10, 0, 0, 0)^\top$, the cost functional $J_N(x_0, u) = \sum_{i = 0}^{N - 1} \int_{t_i}^{t_{i + 1}} | x_5(t) - x_{5, \text{ref}}(t) | dt$ where the sampling instances are equidistantly fixed via $t_i = t_0 + i T$ with $T = 0.2$, set tolerance levels of both the minimizer and the differential equation solver to $10^{-6}$ and use reference function
\begin{align*}
	\zeta_{\text{ref}}(t) = \left\{ \begin{array}{ll} 10, & t \in [0, 5) \cup [9, 10) \\ 0, & t \in [5, 9) \cup [10, 15) \end{array} \right.,
\end{align*}
cf. Figure \ref{Example:Figure:adaptivity:reference function}(a).
\begin{figure}[!ht]
	\subfloat[Tracking function $\zeta_{\text{ref}}(\cdot)$ (solid) and redesign reference $x_{5, \text{ref}}(\cdot)$ (dashed)]{%
		\includegraphics[width=0.45\textwidth]{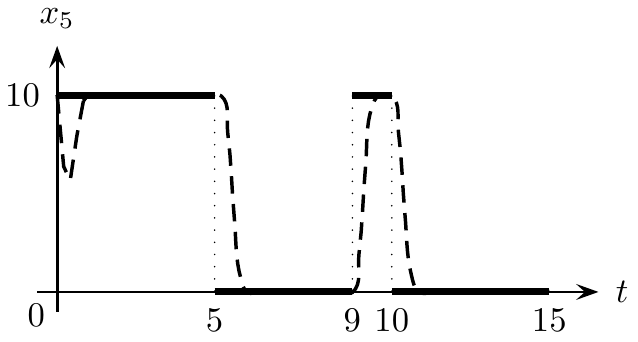}%
	} \hfill
	\subfloat[Horizons for adaptive NMPC with a posteriori estimate (solid) and adaptive NMPC with a priori estimate (dashed)]{%
		\includegraphics[width=0.45\textwidth]{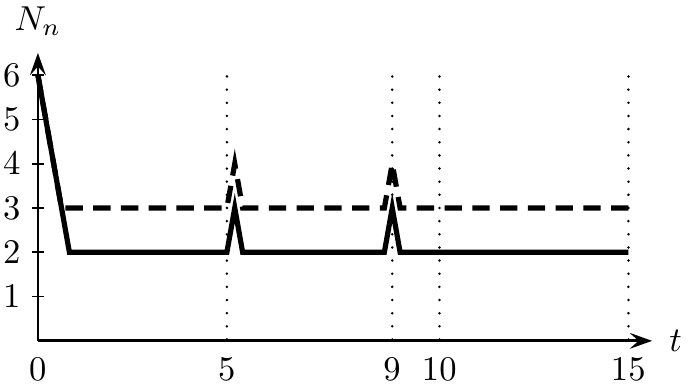}%
	}
	\caption{Tracking function and horizons of the adaptation scheme}
	\label{Example:Figure:adaptivity:reference function}
\end{figure}
For this problem it is known from \cite[Section 5]{GP2009} that the suboptimality estimates of Proposition \ref{Preliminaries:prop:trajectory a posteriori estimate} and Theorem \ref{Preliminaries:thm:apriori variante2} show only poor performance. For this reason we employ the practical variants \cite[Proposition 14 and Theorem 21]{GP2009} with $\varepsilon = 10^{-5}$ and set the lower bound $\overline{\alpha} = 0.5$.\\
Using the a posteriori and a priori estimation techniques within Algorithm \ref{ANMPC:alg:algorithm}, we obtain the evolutions of horizons $N_n$ along the closed loop for the suboptimality bound $\overline{\alpha} = 0.5$ as displayed in Figure \ref{Example:Figure:adaptivity:reference function}(b). In particular, one observes that the less conservative a posteriori algorithms yields smaller optimization horizons which makes the resulting scheme computationally more efficient, see also Table \ref{Example:Table:results adaptive all:zeta}. However, the evaluation of the a posteriori criterion itself is computationally more demanding, see also Figure \ref{Example:Figure:computing time}, below.
\begin{table}[!ht]
	\begin{center}
	\begin{small}
		\begin{tabular}{|c|c||r|r|r||c|c|l|}
			\hline
			\multicolumn{2}{|c||}{Adaptive NMPC} & \multicolumn{3}{c||}{Time in $[10^{-3}s]$} & \multicolumn{3}{c|}{Horizon length} \\
			\multicolumn{1}{|c|}{Implementation} & \multicolumn{1}{c||}{Estimate} & \multicolumn{1}{c|}{$\max$} & \multicolumn{1}{c|}{$\min$} & \multicolumn{1}{c||}{$\varnothing$} & \multicolumn{1}{c|}{$\max$} & \multicolumn{1}{c|}{$\min$} & \multicolumn{1}{c|}{$\varnothing$} \\
			\hline
			Standard NMPC & --- & 									86.50 & 3.23 & 23.14 & 6 & 6 & 6 \\
			Theorem \ref{ANMPC:thm:stepsize prolongation} & a posteriori & 		183.19 & 1.30 & 13.37 & 6 & 2 & 2.39 \\
			Theorem \ref{ANMPC:thm:fixed point} & a posteriori & 			184.79 & 1.49 & 11.55 & 6 & 2 & 2.25 \\
			Theorem \ref{ANMPC:thm:monotone iteration} & a posteriori & 			183.16 & 1.48 & 16.28 & 6 & 2 & 2.43 \\
			Theorem \ref{ANMPC:thm:stepsize prolongation} & a priori & 	226.10 & 2.12 & 17.46 & 6 & 3 & 3.21 \\
			Theorem \ref{ANMPC:thm:fixed point} & a priori & 		219.32 & 1.45 & 18.96 & 6 & 3 & 3.19 \\
			Theorem \ref{ANMPC:thm:monotone iteration} & a priori & 			219.44 & 1.47 & 15.60 & 6 & 3 & 3.13 \\
			\hline
		\end{tabular}
		\caption{Comparison of NMPC results in the tracking type example}
		\label{Example:Table:results adaptive all:zeta}
	\end{small}
	\end{center}
\end{table}

It is also interesting to compare these horizons to the standard NMPC Algorithm with fixed $N$ which needs a horizon of $N=6$ in order to guarantee $\alpha \geq \overline{\alpha}$ along the closed loop. Here, one observes that the required horizon $N_n$ for the adaptive NMPC approach is typically smaller than $N = 6$ for both the a posteriori and the a priori estimate based variant. From Figure \ref{Example:Figure:adaptivity:reference function}(b) one also observes that the horizon is increased at the jump points of the reference function $(\cdot)$, which is the behavior one would expect in a ``critical'' situation and nicely reflects the ability of the adaptive horizon algorithm to adapt to the new situation.\\
Although the algorithm chooses to modify the horizon length throughout the run of the closed loop, one can barely see a difference between the resulting $x_5(\cdot)$ trajectories and the (dashed) reference trajectory given in Figure \ref{Example:Figure:adaptivity:reference function}(a). For this reason, we do not display the closed loop solutions. Instead, we additionally plotted the computing times of the two adaptive NMPC variants in Figure \ref{Example:Figure:computing time}. Again, one can immediately see the spikes in the graph right at the points in which $\zeta_{\text{ref}}(\cdot)$ jumps. This figure also illustrates the disadvantage of the algorithm of having to solve multiple additional optimal control problems whenever $N_n$ is increased, which clearly shows up in the higher computation times at these points, in particular for the computationally more expensive a posteriori estimate. 
\begin{figure}[!ht]
	\begin{center}
		\includegraphics[width=0.59\textwidth]{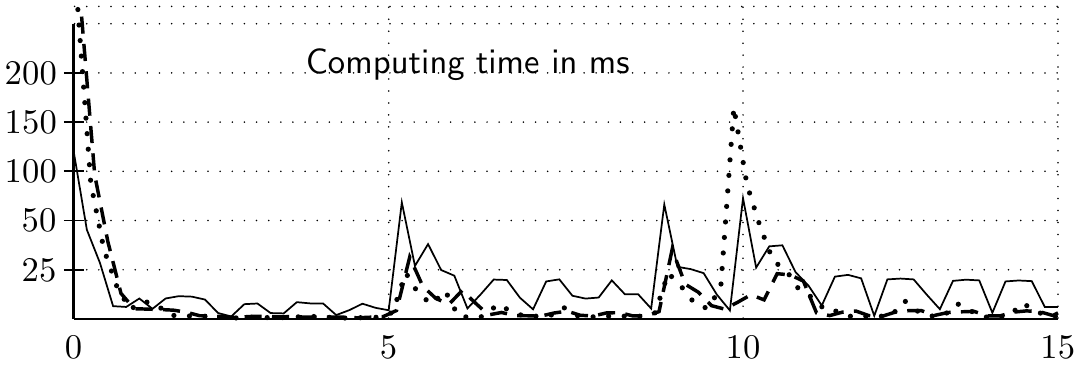}	
		\caption{Computing times of the standard NMPC (solid), adaptive NMPC with a posteriori estimate (dotted) and adaptive NMPC with a priori estimate (dashed)}
		\label{Example:Figure:computing time}
	\end{center}
\end{figure}

Last, we can use Theorem \ref{ANMPC:thm:stability of adaptive NMPC} to obtain the closed loop suboptimality degree $\alpha_C$ of the different implementations shown in Table \ref{Example:Table:results adaptive parameter tracking}. Note that we can restrict ourselves to those time instances where $l(\cdot, \cdot) > \varepsilon$ holds and that according to Theorem \ref{ANMPC:thm:stability of adaptive NMPC} $\alpha_C$ is given by combinations of $C_l(n)$ and $C_\alpha(n)$ while Table \ref{Example:Table:results adaptive parameter tracking} shows minimal and maximal values of $C_l$ and $C_\alpha$. 
\begin{table}[!ht]
	\begin{center}
	\begin{small}
		\begin{tabular}{|c|c||r|r||r|r||r|}
			\hline
			\multicolumn{2}{|c||}{Adaptive NMPC} & \multicolumn{2}{c||}{$C_l$} & \multicolumn{2}{c||}{$C_\alpha$} & \multicolumn{1}{c|}{$\alpha_C$} \\
			\multicolumn{1}{|c|}{Implementation} & \multicolumn{1}{c||}{Estimate} & \multicolumn{1}{c|}{$\min$} & \multicolumn{1}{c||}{$\max$} & \multicolumn{1}{c|}{$\min$} & \multicolumn{1}{c||}{$\max$} & \\
			\hline
			Theorem \ref{ANMPC:thm:stepsize prolongation} & a posteriori 		& 0.9959 &    1.1282 &    1.0000 &    1.0506 &    0.4431 \\
			Theorem \ref{ANMPC:thm:fixed point} & a posteriori			& 1.1078 &    1.1078 &    1.0000 &    1.0000 &    0.4513 \\
			Theorem \ref{ANMPC:thm:monotone iteration} & a posteriori			& 0.9638 &    1.1760 &    1.0000 &    1.1088 &    0.4291 \\
			Theorem \ref{ANMPC:thm:stepsize prolongation} & a priori 		& 0.8824 &    1.1460 &    1.0010 &    1.8321 &    0.4751 \\
			Theorem \ref{ANMPC:thm:fixed point} & a priori			& 0.9480 &    1.5334 &    0.9816 &    1.6457 &    0.3366 \\
			Theorem \ref{ANMPC:thm:monotone iteration} & a priori 			& 0.8835 &    2.0362 &    0.9856 &    1.8360 &    0.2535 \\
			\hline
		\end{tabular}
		\caption{Values of $C_l$, $C_\alpha$ and $\alpha_C$ of Theorem \ref{ANMPC:thm:stability of adaptive NMPC}}
		\label{Example:Table:results adaptive parameter tracking}
	\end{small}
	\end{center}
\end{table}

From Table \ref{Example:Table:results adaptive parameter tracking} we obtain that $\alpha_C$ may deteriorate if the a priori estimate is used while results based on the a posteriori estimate show that $\alpha_C$ is close to the minimal local suboptimality degree $\overline{\alpha} = 0.5$. In either case, the presented adaptation strategies guarantee stability of the closed loop and show a satisfactoring local and closed loop suboptimality degree.

\section{Conclusion}
\label{Section:Conclusion}

We derived various adaptation strategies for the horizon length of an NMPC controller and showed stability and suboptimality of the resulting closed loop trajectory. Moreover, we have shown the practicability and effectiveness of these methods. Future work concerns many parts of this method. Probably the most important point is to improve the \textit{a priori} estimates from Theorem \ref{Preliminaries:thm:apriori variante2} and \cite[Theorem 20]{GP2009} by a more detailed analysis of the parameter $\hat{N}$ and to develop other efficiently computable suboptimality estimates. Moreover, different feasibility conditions as well as development and investigation of alternatives to prolongate or shorten the optimization horizon will be an issue. In particular, combinations of iterates may allow for further insight of the process under control.\\
This work was supported by DFG Grant Gr1569/12 within the Priority Research Program 1305 and the Leopoldina Fellowship Programme LPDS 2009-36.




\end{document}